\documentclass[12pt]{amsart} 
\usepackage{amssymb,amsmath,eufrak,euscript}

\usepackage{graphicx}

\def\R{\mathbb{R}}
\def\Cdb{\mathbb{C}}
\def\C{\mathbb{C}}
\def\D{\mathbb{D}}
\def\Ddb{\mathbb{D}}

\def\H{{\mathcal{H}}}
\begin{document}
\newtheorem{Thm}{Theorem}[section]
\newtheorem{Main}{Main Theorem}
\renewcommand{\theMain}{}
\newtheorem{Cor}[Thm]{Corollary}
\newtheorem{Lem}[Thm]{Lemma}
\newtheorem{Prop}[Thm]{Proposition}
\newtheorem{Def}[Thm]{Definition}
\newtheorem{remark}[Thm]{Remark}

\newtheorem{Obs}[Thm]{Observation}
\def\theequation{\thesection.\arabic{equation}}

\newcommand{\norm}[1]{\Vert#1\Vert}

\title{On functional calculus properties of Ritt operators}

\author{Florence Lancien}

\address{Universit\'e de Franche-Comt\'e, Laboratoire de Math\'ematiques UMR 6623,
16 route de Gray, 25030 Besan\c con Cedex, FRANCE.}
\email{florence.lancien@univ-fcomte.fr}

\author{Christian Le Merdy}

\address{Universit\'e de Franche-Comt\'e, Laboratoire de Math\'ematiques UMR 6623,
16 route de Gray, 25030 Besan\c con Cedex, FRANCE.}
\email{christian.lemerdy@univ-fcomte.fr}

\date{\today}

\thanks{The second named author is supported by the research program ANR 2011 BS01 008 01}

\begin{abstract} We compare various functional calculus properties
of Ritt operators. We show the existence of a Ritt operator $T\colon X\to X$
on some Banach space $X$ with the following property:
$T$ has a bounded $\H^\infty$ functional calculus 
with respect to the unit disc $\D$ (that is, $T$ is polynomially bounded)
but $T$ does not have any bounded $\H^\infty$ functional calculus 
with respect to a Stolz domain of $\D$ with vertex at $1$. Also we show
that for an $R$-Ritt operator, the unconditional Ritt condition
of Kalton-Portal is equivalent to the existence of a bounded $\H^\infty$ functional calculus 
with respect to such a Stolz domain.
\end{abstract}

\maketitle


\bigskip
\section{Introduction}

Ritt operators on Banach spaces have a specific $\H^\infty$ functional calculus
which was formally introduced in \cite{LM}. This functional calculus is related to
various classical notions playing a role in the harmonic analysis of single operators, such as
square functions, maximal inequalities, multipliers and dilation properties, see
in particular the above mentioned paper and \cite{A, ALM, LMX}. The purpose of the present
paper is to compare the $\H^\infty$ functional calculus of Ritt operators to two
closely related notions, namely polynomial boundedness and the unconditional Ritt
condition from \cite{KP}.

Let $\Ddb=\{z\in\Cdb\, :\, \vert z\vert < 1\}$ be the open unit disc of the complex field,
let $X$ be a (complex) Banach space and recall that a bounded operator $T\colon X\to X$
is called polynomially bounded if there exists a constant $K\geq 0$ such that
$$
\Vert P(T)\Vert\,\leq K\sup\bigl\{\vert P(z)\vert\, :\, z\in\Ddb\bigr\}
$$
for any polynomial $P$. We say that $T$ is a Ritt operator provided that
the spectrum of $T$ is included in $\overline{\Ddb}$ and the set
\begin{equation}\label{Ritt}
\bigl\{(\lambda-1)R(\lambda,T)\ :\, |\lambda|>1\bigr\}
\end{equation}
is bounded. (Here $R(\lambda,T)=(\lambda -T)^{-1}$ denotes the resolvent operator.)
For any $\gamma\in \bigl(0,\frac{\pi}{2}\bigr)$, let $B_\gamma$
be the open Stolz domain defined as the interior of the convex hull of $1$ and
the disc $D(0,\sin \gamma)$, see Figure 1 below.

\begin{figure}[ht]
\vspace*{2ex}
\begin{center}
\includegraphics[scale=0.4]{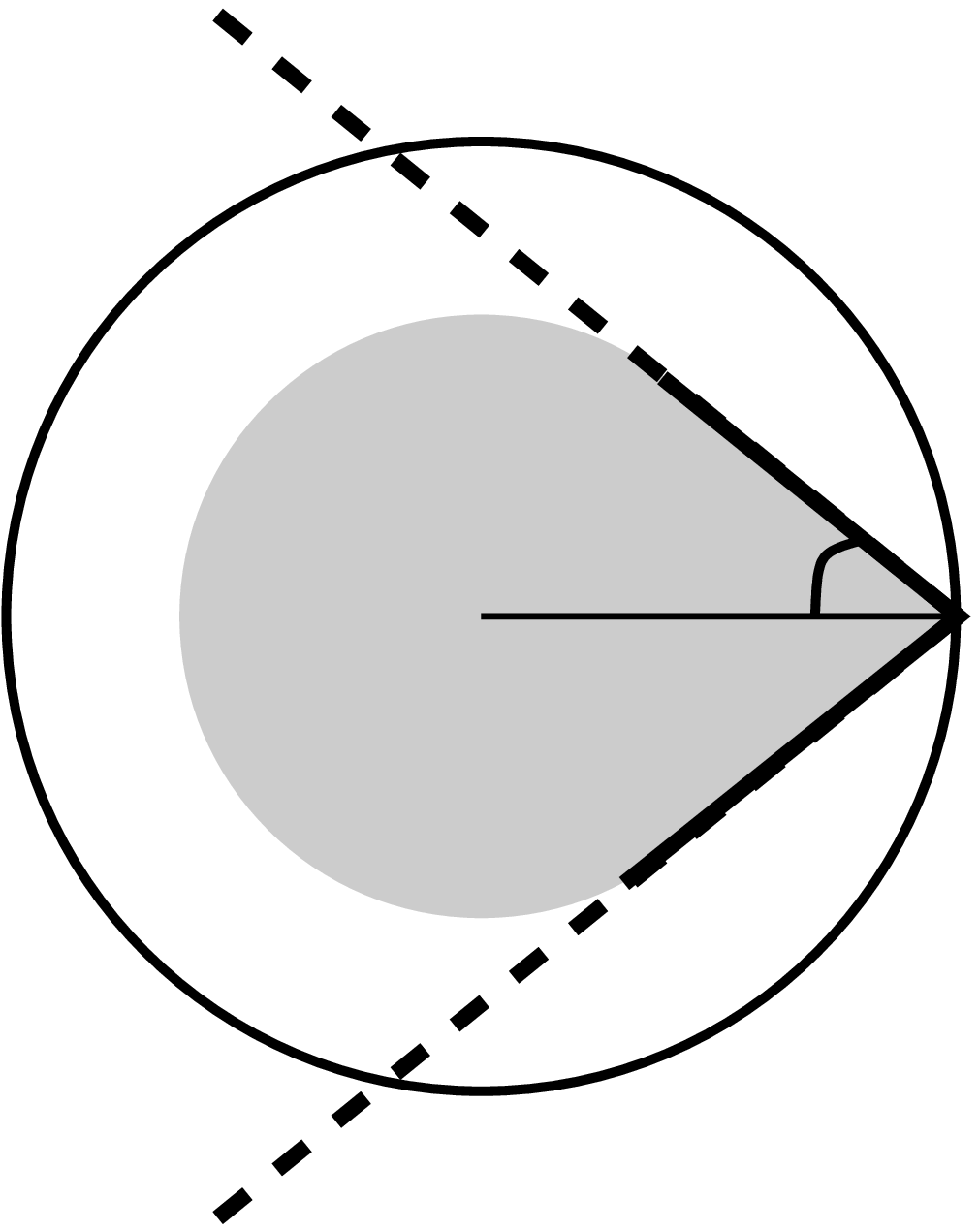}
\begin{picture}(0,0)
\put(-2,65){{\footnotesize $1$}}
\put(-68,65){{\footnotesize $0$}}
\put(-32,81){{\footnotesize $\gamma$}}
\put(-55,85){{\small $B_\gamma$}}
\end{picture}
\end{center}
\caption{\label{f1}}
\end{figure}

It is well-known that the spectrum of any Ritt operator $T$ is included in the closure
$\overline{B_\gamma}$ of one of those Stolz domains. Following \cite{LM}, we say that
$T$ has a bounded $\H^\infty(B_\gamma)$ functional calculus
if there is a constant $K\geq 0$ such that
\begin{equation}\label{bfc}
\Vert P(T)\Vert\,\leq K\sup\bigl\{\vert P(z)\vert\, :\, z\in B_\gamma\bigr\}
\end{equation}
for any polynomial $P$. Since $B_\gamma\subset\Ddb$, it is plain that
this property implies polynomial
boundedness. It was shown in \cite{LM} that the converse holds true on Hilbert spaces.
Our main result asserts that this does not remain true on all Banach spaces.
We will exhibit a Banach space $X$ and
a Ritt operator $T\colon X\to X$ which
is polynomially bounded but has no bounded
$\H^\infty(B_\gamma)$ functional calculus. This will be achieved
in Section 3 (see Theorem \ref{main}). This example
is obtained by first developing and then exploiting a construction of
Kalton concerning sectorial operators \cite{K}. Section 2 is
devoted to preliminary results and to the main features of Kalton's example.

Following \cite{KP} we say that $T$ satisfies the unconditional Ritt condition
if there exists a constant $K\geq 0$ such that
\begin{equation}\label{urc}
\Bigl\Vert\sum_{k\geq 1} a_k\bigl(T^k -T^{k-1})\Bigr\Vert\,\leq K\sup\bigl\{\vert a_k\vert\, :\, k\geq 1\bigr\}
\end{equation}
for any finite sequence $(a_k)_{k\geq 1}$ of complex numbers.
This property is stronger than the Ritt condition \cite[Prop. 4.3]{KP} and it is easy to check that
if $T$ admits a bounded $\H^\infty(B_\gamma)$ functional calculus for some $\gamma<\frac{\pi}{2}$, then
$T$ satisfies the unconditional Ritt condition (see Lemma \ref{fcur} below).
We do not know if the converse holds true. However we will show in Section 4 that
if $T$ is $R$-Ritt and satisfies the unconditional Ritt condition, then it
admits a bounded
$\H^\infty(B_\gamma)$ functional calculus for some $\gamma<\frac{\pi}{2}$.
As a consequence we generalize \cite[Thm. 4.7]{KP} by showing that on a large class of Banach spaces,
the unconditional Ritt condition is equivalent to certain
square function estimates for $R$-Ritt operators.

\section{Sectorial operators and Kalton's example}

Let $X$ be a Banach space and let $A:D(A)\to X$ be a closed operator 
with dense domain $D(A)\subset X$. We let $\sigma(A)$ denote the spectrum of 
$A$ and whenever $\lambda$ belongs to the resolvent set $\Cdb\setminus\sigma(A)$, we let 
$R(\lambda,A)=(\lambda -A)^{-1}$ denote the corresponding resolvent operator.

For any $\omega\in(0,\pi)$, we let $\Sigma_\omega=\{z\in\Cdb^*\, :\, \vert{\rm Arg}(z)\vert<\omega\}$.
We also set $\Sigma_0=(0,\infty)\,$ for convenience. We recall that by definition,
$A$ is sectorial if there exists an angle $\omega$ such that $\sigma(A)\subset\overline{\Sigma_\omega}$ 
and for any $\nu\in(\omega,\pi)$ the set
\begin{equation}\label{sectorial} 
\bigl\{\lambda R(\lambda,A)\ :\, \lambda\in\C\setminus \overline{\Sigma_\nu}\,\bigr\}
\end{equation} 
is bounded. The smallest $\omega\in[0,\pi)$ with this property is called the sectorialy angle 
of $A$.

We will need a few facts about $\H^\infty$ functional calculus for sectorial operators
that we now recall. For backgound and complements, we refer the reader to
\cite{CDMcIY, H, McI}.

Let $A$ be a sectorial operator with sectorialy angle $\omega\geq 0$. 
One can naturally define a bounded operator $F(A)$ for any rational function $F$ 
with nonpositive degree and poles outside $\sigma(A)$. 
Let $\phi\geq\omega$. The operator $A$ is said to admit a 
bounded $\H^\infty(\Sigma_\phi)$ 
functional calculus if
there exists a constant $K$ such that for all functions $F$ as above,
\begin{equation}\label{bfcs} 
\|F(A)\|\leq K \sup\bigl\{|F(z)|\ :\, z\in \Sigma_\phi\bigr\}.
\end{equation} 
In that case, if $\mu$ denotes the infimum of all angles $\phi$ for which such an estimate holds, then 
$A$ is said to admit a bounded $\H^\infty$ functional calculus of type $\mu$.

Note that the above definition makes sense even for $\phi=\omega$, which 
is important for our purpose (see 
Proposition  \ref{Delyon} below).
If $\phi>\omega$ and $A$ has dense range, 
it follows from \cite{CDMcIY, McI} that when the estimate (\ref{bfcs}) 
holds true on rational functions, then the homomorphism $F\mapsto F(A)$ 
naturally extends to a bounded operator on $\H^\infty(\Sigma_\phi)$, the Banach algebra of all bounded 
analytic functions on $\Sigma_\phi.$ In particular for $s\in\R$, the image of the function 
$z\mapsto z^{is}$ under this homomorphism coincides with the classical 
imaginary power $A^{is}$ of $A$. These imaginary powers hence satisfy the estimate 
$$
\| A^{is}\|\leq K e^{\phi |s|},\qquad s\in\R,
$$
when (\ref{bfcs}) holds true.

On a Hilbert space, a well known result of McIntosh \cite{McI} asserts that 
if $A$ is a sectorial operator with sectoriality angle $\omega$ which admits bounded 
imaginary powers or a bounded $\H^\infty(\Sigma_\phi)$ functionnal calculus for some $\phi>\omega$, 
then it has a bounded $\H^\infty(\Sigma_\phi)$ functionnal calculus for any $\phi>\omega$. That is, its 
$\H^\infty$ functional calculus type coincides with its sectoriality angle.

However on general Banach spaces, this property can fail. 
Indeed in \cite{K} Kalton constructs, for any $\theta\in(0,\pi)$, a Banach space 
$X_\theta$ and a sectorial operator $A$ on $X_\theta$ with sectoriality angle $0$, which 
admits a bounded $\H^\infty$ functional calculus of type $\theta$.

The construction is as follows. 
On the classical space $L^2(\R)$, consider 
the norms $\|.\|_\theta$ defined by 
\begin{equation}\label{CH4}
\|f\|_\theta^2=\int_\R e^{-2\theta|\xi|}|\widehat{f}(\xi)|^2\,d\xi.
\end{equation}
Obviously $\|.\|_0$ is the usual $L^2$-norm and $\|\ \cdotp\|_\theta$ is a smaller norm.
For any $\theta\in(0,\pi)$, we let $H_\theta$ denote the completion of $L^2(\R)$ for the norm $\|\ \cdotp\|_\theta$;
this is a Hilbert space.

Let $A$ be the multiplication operator on $L^2(\R)$ defined by 
$$
Af(x)=e^{-x}f(x).
$$
In the sequel we will keep the same notation to denote various
extensions of $A$ on some spaces containing $L^2(\R)$ as a dense subspace.
Note that for any $\phi>0$ and any $F\in \H^\infty(\Sigma_\phi)$, $F(A)$
is the multiplication operator associated to $x\mapsto F(e^{-x})$.

According to \cite{K}, $A$ extends to a sectorial operator on $H_\theta$
with a bounded $\H^\infty$ functional calculus of type $\theta$. 
This (non-trivial) fact follows from the following observations. 
First, for any $f\in L^2(\R)$, we have $A^{is}f(x)= e^{-isx}f(x)$, hence
\begin{equation}\label{CH1}
\widehat{A^{is}f}(\xi)= \widehat{f}(\xi+s)
\end{equation}
for any $s,\xi$ in $\R$. Second, using the definition of 
$\norm{\ \cdotp}_\theta$, this implies that
\begin{equation}\label{BIP}
\|A^{is}\|_{H_\theta\to H_\theta}\,= e^{\theta|s|},\qquad s\in\R.
\end{equation}
This equality implies, by the  above mentioned result of McIntosh, 
that the operator $A$ on $H_\theta$ admits a bounded $\H^\infty(\Sigma_\phi)$ 
functional calculus for all $\phi>\theta$.

The next step is to construct a new completion $X_\theta$ of $L^2(\R)$ on which 
$A$ has similar
$\H^\infty$ functional calculus properties but a `better' sectoriality angle. 
We will point out some important elements of this construction.
Consider a new norm on $L^2(\R)$ by letting 
\begin{equation}\label{CH5}
\|f\|_{X_\theta}=\sup_{a\in\R}\|f\chi_{(-\infty,a)}\|_\theta.
\end{equation}
Then let $X_\theta$ be the completion of $L^2(\R)$ for this norm. 
Clearly for any $f\in L^2(\R)$, we have
$$
\|f\|_\theta\leq \|f\|_{X_\theta}\leq \|f\|_0.
$$
Thus $L^2(\R)\subset X_\theta\subset H_\theta$ with contractive embeddings. Note that
contrary to $H_\theta$,
$X_\theta$ is not a Hilbert space. Again $A$ extends to a sectorial operator
on $X_\theta$. A key fact is that on this new space, the sectoriality angle of $A$ 
is equal to $0$. This is a consequence of the following computation.
For any $f\in L^2(\R)$ and any $\lambda\in \C\setminus \R_+$,
\begin{equation}\label{CH2}
(\lambda-e^{-x})^{-1} f(x)= \int_\R {\lambda e^{-t}\over (\lambda-e^{-t})^2}\, f(x)\,\chi_{(-\infty,t)}(x)\, dt
\end{equation}
for any $x\in\R$.
If we let $\psi=\arg \lambda$, this implies
$$
\|\lambda R(\lambda,A)f\|_{\theta}\leq \|f\|_{X_\theta}\int_0^\infty|s-e^{i\psi}|^{-2}\, ds.
$$
Applying this with $f\chi_{(-\infty,a)}$ instead of $f$, we deduce a uniform
estimate $\norm{\lambda R(\lambda,A)}_{X_\theta\to X_\theta}\leq K_{\psi}$, which yields 
the desired sectoriality property.

If $m\in L^\infty (\R)$ is such that the multiplication operator $f\mapsto mf$ is bounded on 
$H_\theta$ with norm less than $C_m$, then the same holds true on $X_\theta$, since
$$
\|mf\|_{X_\theta}=\sup_{a\in\R}\|mf\chi_{(-\infty,a)}\|_\theta\leq C_m \|f\|_{X_\theta}.
$$
Since $F(A)$ is such a multiplication operator for any $F\in \H^\infty(\Sigma_\phi)$,
we derive the following.

\begin{Lem}\label{HX}
If $A$ admits a bounded $\H^\infty(\Sigma_\phi)$ functional calculus
on $H_\theta$, then it admits a bounded $\H^\infty(\Sigma_\phi)$ functional calculus
on $X_\theta$ as well.
\end{Lem}

Finally, and this is the most difficult part of \cite{K}, it turns out that 
the imaginary powers of $A$ have the same norms on $X_\theta$ and on $H_\theta$, namely 
\begin{equation}\label{CH9}
\|A^{is}\|_{X_\theta\to X_\theta}= \|A^{is}\|_{H_\theta\to H_\theta}= e^{\theta|s|}
\end{equation}
for any $s\in\R$.
Combining with Lemma \ref{HX}, this implies that on $X_\theta$, the operator $A$ admits 
a bounded $\H^\infty(\Sigma_\phi)$ functional calculus for any $\phi>\theta$
but cannot   have
a bounded $\H^\infty(\Sigma_\phi)$ functional calculus for some $\phi<\theta$.

We finally consider the case $\phi=\theta$, which is not treated in \cite{K} but is important 
for our purpose. This requires a new ingredient, namely the next statement 
which is implicit in \cite{LM}.

\begin{Prop}\label{Delyon} 
Let $A$ be a sectorial operator with dense range on some Hilbert space $H$, assume
that $A$ admits bounded imaginary powers and 
that for some $\theta\in(0,\pi)$, they satisfy 
an exact estimate $\|A^{is}\|\leq e^{\theta|s|}$ for any $s\in\R$.
Then $A$ has a bounded $\H^\infty(\Sigma_\theta)$ functional calculus.
\end{Prop}

\begin{proof} 
Let $iU$ be the generator of the $c_0$-semigroup $( A^{is})_{s\geq 0}$. 
Our assumption ensures that it both satisfies
$$
\|e^{s(iU-\theta)}\|\leq 1\qquad 
\textrm{ and }\qquad
\|e^{s(-iU-\theta)}\|\leq 1
$$
for any $s\geq 0$. This means that $iU - \theta$ and $-iU-\theta$ both
generate contractive semigroups on $H$. Thus for all $h\in D(U)$, one has
$$
\textrm{Re} \bigl<(\theta +iU)h,h\bigr>\geq 0
\qquad\textrm{ and }\qquad
\textrm{Re} \bigl<(\theta-iU) h , h\bigr>\geq 0.
$$
Hence the numerical range of $U$ lies in the closed band 
$\Omega=\{z\in\C : |\textrm{Im} z|\leq \theta\}.$ 
By \cite[Thm. 1]{CD}, this implies the existence of a constant $K>0$
such that 
\begin{equation}\label{CH3}
\|G(U)\|\leq K \sup \bigl\{|G(w)| : w\in\Omega\bigr\}
\end{equation}
for any rational function $G$ bounded on $\Omega$. The argument 
in \cite{CD} can be extended to more general functions. It is observed in
\cite{LM} that in particular, it applies to 
all functions $G$ of the form $G(w)= F(e^w)$, where 
$F$ is a rational function with negative degree and poles off 
$\overline{\Sigma_\theta}$
and in this case, $G(U)=F(A)$. 
In this situation, $\sup \{|G(w)| : w\in\Omega\}$
coincides with $\sup \{|F(z)| : z\in\Sigma_\theta\}$. 
Hence we deduce from (\ref{CH3}) that $A$ admits a bounded $\H^\infty(\Sigma_{\theta})$ 
functional calculus.
\end{proof}

According to (\ref{BIP}), the above proposition applies to Kalton's 
operator $A$ on $H_\theta$. Hence the latter 
admits a bounded $\H^\infty(\Sigma_\theta)$ functional calculus. 
Applying Lemma \ref{HX}, we deduce 
that the operator $A$ constructed above on $X_\theta$
has a bounded $\H^\infty(\Sigma_\phi)$ functional calculus for all $\phi\geq\theta$ 
(not only for $\phi>\theta$).

\section{Main result}

Our main purpose is to prove Theorem \ref{main} below.
We first need to modify Kalton's example discussed in the previous section.
Roughly speaking we need a similar example with the additional property that the 
the operator should be bounded. We will get a more precise result.

We consider the restriction $B$ of $A$ on $L^2(\R_+)$.
More explicitly, $B\colon L^2(\R_+)\to L^2(\R_+)$ is the bounded operator
defined by 
$$
Bf(x)=e^{-x}f(x),\qquad f\in L^2(\R_+).
$$
Then we let $H_\theta^+$ be the completion of $L^2(\R_+)$ 
for the norm $\|\ \cdotp \|_{\theta}$ defined by (\ref{CH4}), we let $X_\theta^+$ be the completion of $L^2(\R_+)$ 
for the norm $\|\ \cdotp\|_{X_\theta}$ defined by (\ref{CH5}) and we consider extensions of $B$ to those spaces,
as was done in Section 2. Of course $X_\theta^+$ is a closed subspace of $X_\theta$ and the operator
$B$ on $X_\theta^+$ is the restriction of the operator $A$ on $X_\theta$. Thus for any $\phi\in (0,\pi)$
and any appropriate $F\in \H^\infty(\Sigma_\phi)$, we have $F(B)=F(A)_{\vert X_\theta^+\to X_\theta^+}$, and hence
\begin{equation}\label{CH6}
\norm{F(B)}_{X_\theta^+\to X_\theta^+}\leq\norm{F(A)}_{X_\theta\to X_\theta}.
\end{equation}
Similar comments apply for $H_\theta$ and $H_\theta^+$.

\begin{Prop}\label{Ka} On the Banach space $X_\theta^+$, the operator $B$ is sectorial, 
its sectoriality angle is equal to $0$, 
its spectrum $\sigma(B)$ lies in $[0,1]$, 
it admits a bounded $\H^\infty(\Sigma_\phi)$ 
functional calculus for all $\phi\geq\theta$, 
and
\begin{equation}\label{CH7}
\norm{B^{is}}_{X_\theta^+\to X_\theta^+} =e^{\theta\vert s\vert},\qquad s\in\R.
\end{equation}
\end{Prop}

\begin{proof} 
It is clear from (\ref{CH6}) and results established for $A$ in Section 2 that 
on $X_\theta^+$, $B$ is sectorial with a sectoriality angle equal to $0$, 
and it admits a bounded $\H^\infty(\Sigma_\phi)$ 
functional calculus for all $\phi\geq\theta$.

To show the spectral inclusion $\sigma(B)\subset[0,1]$, 
consider $\lambda\in\Cdb\setminus[0,1]$. As in (\ref{CH2}), we have
$$
(\lambda- e^{-x})^{-1} f(x)=\int_{0}^{\infty} \frac{e^{-t}}{(\lambda-e^{-t})^2}\,
f(x) \,\chi_{(-\infty,t)}(x)\, dt
$$
for any $f\in L^2(\R_+)$ and any $x\geq 0$. Note that contrary to (\ref{CH2}),
integration is now taken on $(0,\infty)$.
We can therefore deduce that 
$$
\norm{(\lambda-B)^{-1}f}_{X_\theta}\leq\,\norm{f}_{X_\theta}
\int_{0}^{\infty}\frac{e^{-t}}{\vert \lambda -e^{-t}\vert^2}\, dt
$$
for any $f\in L^2(\R_+)$, which ensures that $\lambda-B$ is invertible on $X_\theta^+$.

It remains to prove (\ref{CH7}). We will 
establish it by appealing to (\ref{CH9}) and by showing that
for any $s\in\R$,
$$
\|B^{is}\|_{X_\theta^+ \to X_\theta^+}=\|A^{is}\|_{X_\theta\to X_\theta}.
$$
Let us start with a simple observation. Let $\tau_a$ denote the translation operator
defined by $\tau_af(x)=f(x-a)$. Then 
for any $f\in L^2(\R)$ and for any $a\in\R$, we have 
$\widehat{\tau_a f}(\xi)=e^{-ia\xi}\widehat{f}(\xi)$ for any $\xi\in\R$. 
Looking at the definition (\ref{CH4}), we deduce that
\begin{equation}\label{CH8}
\|\tau_af\|_\theta=\| f\|_\theta.
\end{equation}
For any $t\in\R$, we have $\chi_{(-\infty,t)}\tau_a f = \tau_a\bigl(\chi_{(-\infty,t-a)}f\bigr)$
hence we immediately deduce that
\begin{equation}\label{CH8bis}
\|\tau_af\|_{X_\theta}=\| f\|_{X_\theta}.
\end{equation}

Now take a function $f$ in $L^2(\R)$ with 
bounded support included in some compact interval $[-M,M]$. Given
any $t\in \R$, we have
\begin{align*}
\|\chi_{(-\infty,t)} A^{is}f\|_\theta & =\|\tau_M(\chi_{(-\infty,t)} A^{is}f)\|_\theta
\\
& =\|\chi_{(-\infty,t+M)}\tau_M (A^{is}f)\|_\theta
\\
& \leq \|\tau_M (A^{is}f)\|_{X_\theta}
\end{align*}
by  (\ref{CH8}). Further, $A^{is}f(x)=e^{-isx}f(x)$ hence $\bigl[\tau_M (A^{is}f)](x) = e^{isM}A^{is}(\tau_M f)(x)$
for any real $x$. Thus 
$$
\|\tau_M (A^{is}f)\|_{X_\theta} = \| A^{is}(\tau_M f))\|_{X_\theta}.
$$
Since $\tau_M f $ has support in $\R_+$, we derive that
$$
\|\tau_M (A^{is}f)\|_{X_\theta}\leq \norm{B^{is}}_{X_\theta^+\to X_\theta^+}\norm{\tau_M f}_{X_\theta}.
$$
According to (\ref{CH8bis}) and the preceding inequalities, we deduce that
$$
\|\chi_{(-\infty,t)} A^{is}f\|_\theta
\leq \norm{B^{is}}_{X_\theta^+\to X_\theta^+}\norm{f}_{X_\theta}.
$$
Taking the supremum over $t\in\R$, one obtains 
$\|A^{is}f\|_{X_\theta} \leq \norm{B^{is}}_{X_\theta^+\to X_\theta^+}\norm{f}_{X_\theta}$.
Hence 
$$
\|A^{is}\|_{X_\theta\to X_\theta}\leq \|B^{is}\|_{X_\theta^+\to X_\theta^+}.
$$ 
The reverse inequality is clear, see (\ref{CH6}).
\end{proof}

We now turn to Ritt operators. Recall the definition of a bounded $\H^\infty(B_\gamma)$ 
functional calculus from Section 1 (see also \cite{LM}).

\begin{Thm}\label{main} 
There exists a Ritt operator $T$ on a Banach space $X$ which is 
polynomially bounded but admits no bounded $\H^\infty(B_\gamma)$ 
functional calculus for any 
$\gamma<{\pi\over 2}.$
\end{Thm}

\begin{proof}  
We take for $X$ the Banach space $X_{\frac{\pi}{2}}^+$ considered above and we let 
$B\colon X\to X$ be the operator considered in Proposition \ref{Ka}.
Then we let 
$$
T=(I-B)(I+B)^{-1}.
$$
We note that $z\mapsto {1-z\over 1+z}$ maps $\Sigma_{\pi\over 2}$ onto $\D$ and $[0,1]$ into itself. 
Thus 
$$
\sigma(T)\subset[0,1].
$$
To show that $T$ is a Ritt operator, we consider 
$\lambda\in\C$ with $|\lambda |>1$. One can write $\lambda={1-z\over 1+z}$ with 
$z\notin \overline{\Sigma_{\pi\over 2}}$. 
It is easy to check that 
$$
(\lambda-1)(\lambda-T)^{-1}=z(z-B)^{-1}(I+B).
$$
Since the sectorial angle of $B$ is $0$, the set 
$\bigl\{z(z-B)^{-1}\, :\, z\notin \overline{\Sigma_{\pi\over 2}}\bigr\}$ 
is bounded. 
Since $B$ is bounded, we deduce that the set defined in 
(\ref{Ritt}) is bounded.

The fact that $B$ has a bounded $\H^\infty(\Sigma_{\pi\over 2})$ functional calculus on $X$ 
implies that $T$ is polynomially bounded. Indeed if $P$ is a polynomial, 
then $P(T)=F(B)$ for the rational function $F$ defined by $F(z)=P\bigl({1-z\over 1+z}\bigr)$. Hence
for some constant $K$, we have
$$
\|P(T)\|= \|f(B)\|\leq K \sup \bigl\{|F(z)|: z\in \Sigma_{\pi\over 2}\bigr\},
$$
and moreover,
$$
\sup \bigl\{|F(z)|: z\in \Sigma_{\pi\over 2}\bigr\} = \sup \bigl\{|P({w})|: w\in \D\bigr\}.
$$

Now assume that $T$ has a bounded $\H^\infty(B_\gamma)$ 
functional calculus for some
$\gamma<{\pi\over 2}.$ Consider the auxiliary operator
$$
C= I-T = 2B(I+B)^{-1}.
$$
By \cite[Prop. 4.1]{LM}, $C$ is a sectorial operator 
which admits a bounded $\H^\infty(\Sigma_\theta)$ for some $\theta\in (0,\frac{\pi}{2})$. Thus
there exists a constant $K>0$ such that  
$$
\|C^{is}\|\leq K e^{\theta |s|},\qquad s\in\R.
$$
Further $\sigma(I+B)\subset[1,2]$. Thus $I+B$ is bounded and invertible and hence
it admits a bounded $\H^\infty$ functional calculus of any type. Thus for any $\theta'>0$. 
there exists 
$K'>0$ such that  
$$
\|(I+B)^{is}\|\leq K' e^{\theta'|s|}.
$$
Since $B$ and $C$ commute, we have
$$
B^{is} = 2^{-is} C^{is}(I+B)^{is},
$$
hence
$$
\|B^{is}\|\leq KK^{'} e^{{(\theta+\theta')}|s|}
$$
for any $s\in\R$. Applying this with $\theta'$ small enough so that $\theta+\theta'<\frac{\pi}{2}$,
this contradicts (\ref{CH7}) on $X_{\frac{\pi}{2}}^{+}$.
\end{proof}

\begin{remark}\label{RR}{\rm 
A Ritt operator $T$ on Banach space $X$ is called $R$-Ritt if the bounded set in (\ref{Ritt}) is actually
$R$-bounded. That notion was introduced in \cite{Bl1}, in relation
with the study of discrete maximal regularity, see also \cite{Bl2, KP, LM, P}. Background and references
on $R$-boundedness can also be found in the latter references.

The existence of Ritt operators which are not $R$-Ritt goes back to Portal \cite{P}.
According to \cite[Prop. 7.6]{LM}, a polynomially bounded $R$-Ritt operator has a bounded
$\H^\infty(B_\gamma)$ functional calculus for some $\gamma<\frac{\pi}{2}$. Thus
the operator $T$ constructed in Theorem \ref{main} is a Ritt operator which is not $R$-Ritt.
This example is of a different nature than the ones from \cite{P}.}
\end{remark}

\section{Unconditional Ritt operators}

We now investigate the links  between the unconditional Ritt condition and the $\H^\infty$ functional calculus. 
It is observed in \cite{KP} that the 
unconditional Ritt condition (\ref{urc}) is equivalent to the existence of a constant $K>0$ such that
\begin{equation}\label{urc1}
\sum_{k\geq 1} \bigl\vert\bigl\langle
\bigl(T^k -T^{k-1})x,y\bigr\rangle\bigr\vert\,\leq \,K\|x\| \|y\|,
\qquad x\in X, y\in X^*.
\end{equation}
Moreover it is stronger than the Ritt condition.  
We will now show that the unconditional Ritt condition is 
weaker than the existence of a bounded $\H^\infty(B_\gamma)$ functional 
calculus for some $\gamma<\frac{\pi}{2}$.

\begin{Lem}\label{fcur} If $T$ admits a 
bounded $\H^\infty(B_\gamma)$ functional calculus for some 
$\gamma<\frac{\pi}{2}$, then
$T$ satisfies the unconditional Ritt condition.
\end{Lem}

\begin{proof} 
Assume that $T$ admits a bounded $\H^\infty(B_\gamma)$ functional calculus for some 
$\gamma<\frac{\pi}{2}$. Consider a finite sequence $(a_k)_{k\geq 1}$.
Since 
$$
\sum_{k\geq 1} a_k\bigl(T^k -T^{k-1})= P(T)
$$ 
for the polynomial $P$ defined by 
$P(z)=\sum_{k\geq 1} a_k\bigl(z^k -z^{k-1})$,  (\ref{bfc}) implies that 
$$
\Bigl\Vert\sum_{k\geq 1} a_k\bigl(T^k -T^{k-1})\Bigr\Vert\,
\leq K\sup\bigl\{\vert P(z)\vert\, :\, z\in B_\gamma\bigr\}.
$$
Now we have 
$$
|P(z)| \leq \sup_{ k\geq 1} |a_k| \ \sum_{k\geq 1} |z^k -z^{k-1}|=\sup_{ k\geq 1} |a_k|  
\biggl(\frac{|z-1|}{1-|z|}\biggr).
$$ 
Since $z\mapsto \frac{|z-1|}{1-|z|}$ is bounded on 
$B_\gamma$, this implies the unconditional Ritt condition (\ref{urc}).
\end{proof}

We now show a partial converse. See Remark \ref{RR} for the notion of $R$-Ritt operator. We will
use the companion notion of $R$-sectorial operator. We recall that a sectorial operator
$A$ on Banach space is called $R$-sectorial if 
there exists an angle $\omega$ such that $\sigma(A)\subset\overline{\Sigma_\omega}$ 
and for any $\nu\in(\omega,\pi)$ the set (\ref{sectorial}) is $R$-bounded. 
In accordance with terminology in Section 2, the smallest $\omega\in[0,\pi)$ 
with this property will be called the $R$-sectorialy angle 
of $A$. We refer the reader to \cite{Bl1,Bl2,KW,LM} and the references therein for information
on $R$-sectoriality.

\begin{Thm}\label{ThmUnc} Let  $T$ be an $R$-Ritt operator which satisfies the unconditional Ritt condition, then it
admits a bounded
$\H^\infty(B_\gamma)$ functional calculus for some $\gamma<\frac{\pi}{2}$.
\end{Thm}

\begin{proof} We consider the operator 
$$
C=I-T.
$$
According to \cite[Thm. 1.1]{Bl1} and its proof, the assumption that $T$ is $R$-Ritt 
implies that $C$ is $R$-sectorial, with an $R$-sectoriality angle $<\frac{\pi}{2}$.
On the other hand the unconditional Ritt condition (\ref{urc}) for $T$ implies the so-called 
$L_1$-condition for $C$ :
$$
\int_0^\infty \bigl\vert\bigl\langle
Ce^{-tC} x, y\bigr\rangle\bigr\vert\,
\frac{dt}{t}\leq K \|x\|\|y\|,\qquad  x\in X, y\in X^*.
$$
Indeed for any $t>0$,
$$
Ce^{-tC}=(I-T)e^{-t}e^{tT}=\sum_{n\geq 0} (I-T) e^{-t} \,\frac{t^nT^n}{n!}.
$$
Thus for any $x\in X$ and $y\in X^*$, we have
$$
\bigl\langle
Ce^{-tC}x,y\bigr\rangle 
= \sum_{n\geq 0} e^{-t}\,\frac{t^n}{n!}\,\bigl\langle (I-T)T^n x,y\bigr\rangle.
$$
This implies, using (\ref{urc1}), that 
\begin{align*}
\int_0^\infty  \bigl\vert\bigl\langle
Ce^{-tC} x, y\bigr\rangle\bigr\vert\,
\frac{dt}{t} & \leq \sum_{n\geq 0}\frac{1}{n!} 
\int_0^\infty \bigl\vert\bigl\langle (I-T)T^n x,y
\bigr\vert\bigr\rangle\, e^{-t} t^{n-1}\, dt \\
& =\sum_{n\geq 0} \bigl\vert\bigl\langle (I-T)T^n x,y\bigr\rangle\bigr\vert   \\
& \leq K\|x\|\|y\|.
\end{align*}

Now by results of \cite[Section 4]{CDMcIY}, the $L_1$-condition implies 
that $C$ admits a bounded $\H^\infty(\Sigma_\theta)$ functional calculus for all 
$\theta>\frac{\pi}{2}$. Since $C$ is $R$-sectorial with an $R$-sectoriality angle 
$<\frac{\pi}{2}$, it follows from 
\cite[Prop. 5.1]{KW} that $C$ actually admits a bounded $\H^\infty(\Sigma_\theta)$ functional calculus for 
some $\theta<\frac{\pi}{2}$. By \cite[Prop. 4.1]{LM}, this is 
equivalent to the fact that $T$ has a  bounded $\H^\infty(B_\gamma)$ functional 
calculus for some $\gamma<\frac{\pi}{2}$.
\end{proof}

It is shown in \cite[Thm. 4.7]{KP} that when $X$ is a Hilbert space, the unconditional
Ritt condition is equivalent to certain square function estimates. We can now extend
that result to $L^p$-spaces. In the next statement, we let $p'=p/(p-1)$ denote
the conjugate number of $p$.

\begin{Cor}\label{Lp}
Let $\Omega$ be a measure space, let $1<p<\infty$ and let $T\colon L^p(\Omega)\to L^p(\Omega)$
be a power bounded operator. The following assertions are equivalent.
\begin{itemize}
\item [(i)]
$T$ is $R$-Ritt and satisfies the unconditional Ritt condition.
\item [(ii)] There exists a constant $C>0$ such that
\begin{equation}\label{SF1}
\Bigl\Vert\Bigl(\sum_{k=1}^{\infty} k\,\bigl\vert T^{k}(x) -T^{k-1}
(x)\bigr\vert^2\Bigr)^{\frac{1}{2}}
\Bigr\Vert_{p}\,\leq C\Vert x \Vert
\end{equation}
for any $x\in L^p(\Omega)$ and
\begin{equation}\label{SF2}
\Bigl\Vert\Bigl(\sum_{k=1}^{\infty} k\,\bigl\vert T^{*k}(y) -T^{*(k-1)}(y)\bigr\vert^2\Bigr)^{\frac{1}{2}}
\Bigr\Vert_{p'}\,\leq C\Vert y \Vert
\end{equation}
for any $y\in L^{p'}(\Omega)$.
\end{itemize}
\end{Cor}

\begin{proof}
If the square function estimates in (ii) hold true,
then $T$ is an $R$-Ritt operator by \cite[Thm. 5.3]{LM}.
Further $T$ has a bounded $\H^\infty(B_\gamma)$ functional calculus
for some $\gamma < \frac{\pi}{2}$, by \cite[Thm. 1.1]{LM}. Hence
Lemma \ref{fcur} ensures that $T$ satisfies the unconditional Ritt condition.
The converse assertion that (i) implies (ii) is obtained by combining Theorem \ref{ThmUnc}
and \cite[Thm. 1.1]{LM}.
\end{proof}

It is clear from \cite{LM} that Corollary \ref{Lp} holds as well on reflexive
Banach lattices with finite cotype. Further generalizations hold true on more
Banach spaces, using the abstract
square functions introduced and discussed in \cite{LM}, to which we refer for
more information. Combining
the results from that paper with Theorem \ref{ThmUnc}, one obtains that when $X$
has finite cotype and $T\colon X\to X$ is an $R$-Ritt operator, then
$T$ satisfies the unconditional Ritt condition if and only if $T$
and $T^*$ admit square function estimates.

\end{document}